\documentclass[11pt]{amsart}

\usepackage{amsmath,amsthm,amsfonts,amssymb,bm,graphicx }

\usepackage
{hyperref}
\hypersetup{colorlinks=true,citecolor=blue,linkcolor=blue,urlcolor=blue}

\theoremstyle{plain}
\newtheorem{theorem}{Theorem}
\newtheorem{lemma}[theorem]{Lemma}

\newtheorem{prop}[theorem]{Proposition}

\theoremstyle{remark}
\newtheorem*{remark}{Remark}
\newtheorem{conj}[theorem]{Conjecture}

\numberwithin{equation}{section}

\newcommand{\pr}{^\prime}

\newcommand{\ve}{\varepsilon}

\makeatletter
\DeclareRobustCommand\widecheck[1]{{\mathpalette\@widecheck{#1}}}
\def\@widecheck#1#2{%
    \setbox\z@\hbox{\m@th$#1#2$}%
    \setbox\tw@\hbox{\m@th$#1%
       \widehat{%
          \vrule\@width\z@\@height\ht\z@
          \vrule\@height\z@\@width\wd\z@}$}%
    \dp\tw@-\ht\z@
    \@tempdima\ht\z@ \advance\@tempdima2\ht\tw@ \divide\@tempdima\thr@@
    \setbox\tw@\hbox{%
       \raise\@tempdima\hbox{\scalebox{1}[-1]{\lower\@tempdima\box
\tw@}}}%
    {\ooalign{\box\tw@ \cr \box\z@}}}
\makeatother

\begin{document}

\author{V\' \i t\v ezslav Kala}
\address{Department of Algebra, Faculty of Mathematics and Physics, Charles University, Sokolov\-sk\' a 83, 18600 Praha~8, Czech Republic}
\email{vita.kala@gmail.com}

\title{Norms of indecomposable integers in real quadratic fields}

\thanks{The author was supported by Charles University Mobility Fund.}

\begin{abstract} We study totally positive, additively indecomposable integers in a real quadratic field $\mathbb Q(\sqrt D)$.
We estimate the size of the norm of an indecomposable integer by expressing it as a power series in $u_i^{-1}$, where $\sqrt D$ has the periodic continued fraction expansion $[u_0, \overline{u_1, u_2, \dots, u_{s-1}, 2u_0}]$. 
This enables us to disprove a conjecture of Jang-Kim \cite{JK} concerning the maximal size of the norm of an indecomposable integer.
\end{abstract}


\setcounter{tocdepth}{2}  \maketitle 

\section{Introduction} 

A totally positive integer in a real quadratic field $K=\mathbb Q(\sqrt D)$ is (additively) indecomposable if it can't be expressed as the sum of two totally positive integers.
Indecomposable integers can be explicitly described using the continued fraction  
$\sqrt D=[u_0, \overline{u_1, u_2, \dots, u_{s-1}, 2u_0}]$ and are deeply connected to the structure of the number field $\mathbb Q(\sqrt D)$. 
Recently Blomer and the author \cite{BK}, \cite{Ka} investigated the relation between indecomposable elements and universal quadratic forms over $\mathcal O_K$, 
and showed that if there are $M$ indecomposables (satisfying certain additional properties), then
every universal, totally positive definite, quadratic form over $\mathcal O_K$ has at least $M$ variables.

One of the tools was the easy observation \cite[Lemma 3]{BK} that every totally positive integer, which is not divisible by any rational integer and has sufficiently small norm (at most $\sqrt D$), is indecomposable. This can be viewed as a lower bound on the norm that guarantees indecomposability. 
On the other hand, there are only finitely many indecomposables up to multiplication by units, and so there is a maximum of their norms.

The search for such an upper bound on the norm was started by 
Dress and Scharlau \cite{DS} in 1982, when they proved that every indecomposable integer has norm less than or equal to $D$. Their result was recently improved by Jang and Kim \cite{JK}, who showed that in fact the maximum is at most $\frac D{N}$, 
where $N$ is the minimum of absolute values of negative norms of elements of $\mathcal O_K$. (Both of these results can be improved when $D\equiv 1\pmod 4$ and $\mathbb Z[\sqrt D]\neq \mathcal O_K$ -- however, for simplicity we restrict only to the case $D\equiv 2, 3\pmod 4$ in this paper).

Jang and Kim also proved that the upper bound is optimal in some cases and stated a general conjecture concerning an improvement of the bound, which we repeat as Conjecture \ref{conjecture} below.

In this note, we show that the conjecture doesn't hold. As an illustration we first provide a specific counterexample in Theorem \ref{counterexample}. Then we
give algorithms for computing power series expansions for the norms of negative convergents and indecomposables, compute the first few terms explicitly (Theorem \ref{power series ni} and Proposition \ref{power series beta}), and determine which indecomposable elements have large norms (Proposition \ref{heuristic r}). 
These results have guided us towards finding the example in Theorem \ref{counterexample}, but one can also most likely use them to obtain infinitely many counterexamples, as we indicate at the end of the paper. 
This argument and the proof of Theorem \ref{counterexample} involve a fairly routine verification in Mathematica.

\section*{Acknowledgements}

I wish to thank Valentin Blomer for his very useful comments and suggestions, and to Milan Boh\' a\v cek for his help with the Mathematica program.

\section{Conjecture of Jang-Kim}

Throughout this paper, let $D$ be a squarefree positive integer.
We shall work in the real quadratic field $K=\mathbb Q(\sqrt D)$ and its ring of integers $\mathcal O_K$. For simplicity, we always assume that $D\equiv 2, 3\pmod 4$, so that $\mathcal O_K=\mathbb Z[\sqrt D]$ (although the arguments can be easily modified to cover the case $D\equiv 1\pmod 4$ as well).

For $\alpha=x+y\sqrt D\in \mathbb Q(\sqrt D)$ we denote its conjugate by $\alpha\pr=x-y\sqrt D$ and its norm by $N(\alpha)=\alpha\alpha\pr=x^2-Dy^2$.
We write $\alpha\succ\beta$ to denote $\alpha>\beta$ and $\alpha\pr>\beta\pr$, and say that $\alpha$ is totally positive if $\alpha\succ 0$.
An element $\alpha\in\mathcal O_K$ is (additively) indecomposable if there are no $\beta, \gamma\in\mathcal O_K$ such that $\alpha=\beta+\gamma$ and $\beta, \gamma\succ 0$.

We shall use the following notation and its well-known properties (see eg. \cite{HW}, \cite{Pe} for a reference):
\begin{itemize}
\item $\sqrt D=[u_0, \overline{u_1, u_2, \dots, u_{s-1}, 2u_0}]$ is a periodic continued fraction with $D\equiv 2,3\pmod 4$ a squarefree positive integer
\item $s, u_0, u_1, u_2, \dots, u_{s-1}\in\mathbb N$, $u_{si}=2u_0$, and $u_{si+j}=u_j$ for $i>0$ and $j\geq 0$
\item the sequence $(u_1, u_2, \dots, u_{s-1})$ is symmetric, i.e., $u_{s-i}=u_i$ for $1\leq i\leq s-1$
\item $\frac {p_i}{q_i}:=[u_0, \dots, u_i]$ is the $i$th convergent to $\sqrt D$
\item $p_{i+1}=u_{i+1}p_i+p_{i-1}$ and $q_{i+1}=u_{i+1}q_i+q_{i-1}$ 
(with initial conditions $p_{-1}=1$, $p_0=u_0$, $q_{-1}=0$, $q_0=1$)
\item $p_{i+1}q_i-p_iq_{i+1}=(-1)^i$
\item $\alpha_i:=p_i+q_i\sqrt D\in\mathbb Z[\sqrt D]$ for $i\geq -1$
\item $\alpha_i\succ 0$ $\Leftrightarrow$ $i$ is odd 
\item $N_i:=|N(\alpha_i)|=|p_i^2-Dq_i^2|=(-1)^{i+1}N(\alpha_i)$
\item $N:=\min\left\{|N(\alpha)|, \alpha\in\mathbb Z[\sqrt D]\mathrm{\ such\ that\ }N(\alpha)<0\right\}=\min\left\{|N(\alpha_i)|, i \mathrm{\ even}\right\}$ is the minimum of absolute values of negative norms
\item $T_i:=p_ip_{i-1}-Dq_iq_{i-1}$
\item $\alpha_i\alpha_{i+1}\pr=T_{i+1}+(-1)^i\sqrt D$
\item $c_i:=[u_i, u_{i+1}, u_{i+2}, \dots]$, $c_i=u_i+\frac 1{c_{i+1}}$
\item $\sqrt D=\frac{c_{i+1}p_i+p_{i-1}}{c_{i+1}q_i+q_{i-1}}$
\item $\alpha_{i, r}:=\alpha_i+r\alpha_{i+1}$ is a semiconvergent for $i\geq -1$ and $0\leq r\leq u_{i+2}$
\item $\alpha_{i, 0}=\alpha_i$, $\alpha_{i, u_{i+2}}=\alpha_{i+2}$
\item $\alpha_{i, r}\succ 0$ for $0\leq r\leq u_{i+2}$ $\Leftrightarrow$ $i$ is odd 
\item $M_i:=N(\alpha_{i, \lfloor u_{i+2}/2\rfloor})$
\end{itemize}

It is well-known that all the indecomposable integers are (some of) the semiconvergents, see eg. \cite[\S 16 Nebenn\" aherungsbr\" uche]{Pe}.

\begin{prop}\label{characterize indecomposable}
The indecomposable integers of $\mathbb Z[\sqrt D]$ are exactly the semiconvergents $\alpha_{i, r}$, $\alpha_{i, r}\pr$ for odd $i\geq -1$ and $0\leq r<u_{i+2}$.
\end{prop}

We are interested in estimating the maximal norm of an indecomposable integer. First of all, Jang and Kim proved the following interesting result.

\begin{theorem}\cite[Theorem 5]{JK}\label{indecomposable estimate norm}
\nopagebreak

a) Let $i$ be odd. Then $$N(\alpha_{i,r})=\frac{D-(T_{i+1}+rN(\alpha_{i+1}))^2}{|N(\alpha_{i+1})|}.$$

b) If $\alpha\in\mathbb Z[\sqrt D]$ is indecomposable, then $N(\alpha)\leq \frac DN$.
\end{theorem}

\begin{proof}
For the sake of completeness, we briefly indicate the proof, following \cite{JK}. We have
\begin{displaymath}
\begin{split}
N(\alpha_{i, r})&=N(\alpha_i+r\alpha_{i+1})
=N\left(\frac{1}{\alpha_{i+1}\pr}\cdot\alpha_{i+1}\pr(\alpha_i+r\alpha_{i+1})\right)
=\frac{N(\alpha_{i+1}\pr\alpha_i+rN(\alpha_{i+1}))}{N(\alpha_{i+1})}=\\
&= \frac{N(T_{i+1}+rN(\alpha_{i+1})+(-1)^i\sqrt D)}{N(\alpha_{i+1})}=
\frac{(T_{i+1}+rN(\alpha_{i+1}))^2- D}{N(\alpha_{i+1})}=\\
&=\frac{D-(T_{i+1}+rN(\alpha_{i+1}))^2}{|N(\alpha_{i+1})|}\leq \frac{D}{N},
\end{split}
\end{displaymath}
since $|N(\alpha_{i+1})|\geq N$ by the definition of $-N$ as the maximum of negative norms of elements of $\mathbb Z[\sqrt D]$. Together with Proposition \ref{characterize indecomposable}, this immediately implies b).
\end{proof}

Let $i_0$ be an index such that the minimum $N$ of absolute values of negative norms satisfies $N=N_{i_0+1}(=|N(\alpha_{i_0+1})|)$.
Motivated by the preceding Theorem \ref{indecomposable estimate norm}, Jang and Kim expected that the maximum norm of an indecomposable integer is attained at $\alpha_{i, r}$ for $i=i_0$ and some $r$. This expectation then led them to make the following conjecture.

\begin{conj}\cite[Conjecture 1]{JK}\label{conjecture}
Let $a$ be the smallest nonnegative rational integer such that $N$ divides $D-a^2$. Then $N(\alpha)\leq \frac {D-a^2}{N}$ for all indecomposable $\alpha\in\mathbb Z[\sqrt D]$.
\end{conj}

However, the expectation need not be true, i.e., we can have $N(\alpha_{j,t})>N(\alpha_{i_0,r})$ for some $j\neq i_0$, and then the conjecture may not hold. 

\begin{theorem}\label{counterexample}
Let $D=24\,009\,857\,226\,825\,282\,345\,490$. Then: 
\begin{enumerate}
\item $D\equiv 2\pmod 4$ is squarefree and its continued fraction is
$$\sqrt D=[u_0, \overline{10,2,12,6,1,3,4,3,12,3,4,2,1,6,12,2,10,2u_0}]$$
with $u_0=154\,951\,144\,645$.
\item $\alpha_2$ has the largest negative norm $-N=-N_2=-24\,548\,583\,881$
\item $\alpha_{7, 6}$ is the indecomposable integer with the largest norm $M_7=977\,608\,342\,706$
\item The smallest nonnegative rational integer $a$ such that $N$ divides $D-a^2$ is $a=4\,030\,160\,489$.
\item $977\,608\,342\,706=M_7>\frac{D-a^2}{N}=M_1=977\,393\,040\,249$
\end{enumerate}
Hence Conjecture \ref{conjecture} is false over $\mathbb Q(\sqrt D)$.
\end{theorem}

\begin{proof}
These results are easily verified by a computation in Mathematica, the file with the computations is avalaible at \url{sites.google.com/site/vitakala/research/indec}.

(1) is straightforward. For (2), we know that the element with largest negative norm is some $\alpha_i$ with even $i$, $0\leq i\leq s=18$. Hence we just need to check these 10 possibilities.
(Note that the second largest is $\alpha_8$ with norm $-N_8=-24\,559\,791\,665$.)

By Proposition \ref{characterize indecomposable} we know that (up to multiplication by units and conjugation), $\alpha_{i, r}$ with odd $i$, $-1\leq i\leq s-2$ and $0\leq r<u_{i+2}$, are exactly the indecomposable integers. Again this is a small set of values that we need to check to prove (3)
(in fact, we could restrict it even more using Proposition \ref{heuristic r}).

(4) is obtained by solving the congruence $a^2\equiv D\pmod N$, and (5) then follows.
\end{proof}

Our goal in the rest of the paper is to give good estimates on the sizes of $N_{i}$ and $N(\alpha_{i, r})$, to use them to explain how we found the counterexample in Theorem \ref{counterexample}, and then to outline the argument that there are most likely infinitely many of them.

\begin{remark}
Note that the assumption that $\alpha$ is indecomposable is missing from the statement of Conjecture 1 in \cite{JK}. Also, in the discussion immediately preceding the conjecture, there should probably be ``$x+tN(p_{s-1})=a$".
\end{remark}

\section{Norms of convergents}

In this section we give an algorithm for expressing $N_i:=|N(\alpha_i)|$ as a power series in $u_0^{-1}, u_1^{-1},$ $u_2^{-1},\dots,$ $u_{s-1}^{-1}$ and compute the first few terms (Theorem \ref{power series ni}). For this purpose, we first prove recurrence relations for $1/c_i$ and $N_i$.

\begin{prop}\label{recurrence ni}
For $i\geq 0$ we have 

a) $$\frac{1}{c_i}=\frac{1}{u_i}\left(1-\frac{1}{u_ic_{i+1}}+\frac{1}{u_i^2c_{i+1}^2}-\dots\right)=\sum_{j=0}^{\infty}\frac{(-1)^j}{u_i^{j+1}c_{i+1}^j},$$ 

b) $$T_i=(-1)^{i+1}\sqrt D-N(\alpha_i)c_{i+1},$$ 

c) $$N_i=\frac{2\sqrt D}{c_{i+1}}-\frac{N_{i-1}}{c_{i+1}^2},$$ and

d) $$\frac{N_i}{c_{i+2}}<\sqrt D.$$
\end{prop}

\begin{proof}
From the definition of $c_i$ we have
$$c_i=[u_i, u_{i+1}, u_{i+2}, \dots]=u_i+\frac {1}{[u_{i+1}, u_{i+2}, \dots]}=u_i+\frac 1{c_{i+1}}.$$

To prove a), we take the reciprocal of this formula and obtain
$$\frac{1}{c_i}=\frac{1}{u_i+\frac 1{c_{i+1}}}=\frac{1}{u_i}\cdot\frac{1}{1+\frac{1}{u_ic_{i+1}}}=
\frac{1}{u_i}\cdot\sum_{j=0}^{\infty}(-1)^j\frac{1}{u_i^{j}c_{i+1}^j}.$$
Since $u_ic_{i+1}=u_i(u_{i+1}+\frac 1{c_{i+2}})\geq 1\cdot (1+\frac 1{c_{i+2}})>1$, the series converges absolutely.

For the second formula b), note that since 
$\sqrt D=\frac{c_{i+1}p_i+p_{i-1}}{c_{i+1}q_i+q_{i-1}}$,
we have $c_{i+1}(p_i-q_i\sqrt D)=-p_{i-1}+q_{i-1}\sqrt D$. This in turn implies
$$c_{i+1}=\frac{-p_{i-1}+q_{i-1}\sqrt D}{p_i-q_i\sqrt D}=
\frac{(p_i+q_i\sqrt D)(-p_{i-1}+q_{i-1}\sqrt D)}{(p_i+q_i\sqrt D)(p_i-q_i\sqrt D)}=
-\frac{T_i+(-1)^i\sqrt D}{N(\alpha_i)},$$
i.e.,
$T_i=(-1)^{i+1}\sqrt D-N(\alpha_i)c_{i+1}$.

To prove c), we can also express
$$T_{i+1}=p_{i+1}p_i-q_{i+1}q_iD=(u_{i+1}p_i+p_{i-1})p_i-(u_{i+1}q_i+q_{i-1})q_iD=
u_{i+1}N(\alpha_i)+T_i.$$
Plugging in the expression for $T_i$ above, we see that
$$T_{i+1}=(-1)^{i+1}\sqrt D-N(\alpha_i)(c_{i+1}-u_{i+1})=
(-1)^{i+1}\sqrt D-\frac{N(\alpha_i)}{c_{i+2}}.$$

Finally, $$(-1)^{i+1}\sqrt D-N(\alpha_{i})c_{i+1}=T_{i}=
(-1)^{i}\sqrt D-\frac{N(\alpha_{i-1})}{c_{i+1}},$$
and hence
$$N(\alpha_{i})c_{i+1}=
2\cdot (-1)^{i+1}\sqrt D+\frac{N(\alpha_{i-1})}{c_{i+1}}.$$

This finishes the proof of c) using the definition $N_j=(-1)^{j+1}N(\alpha_j)$ for $j=i, i+1$.

d) Using c), we see that
$$\frac{N_i}{c_{i+2}}<\frac{2\sqrt D}{c_{i+1}c_{i+2}}=
\frac{2\sqrt D}{(u_{i+1}+1/c_{i+2})c_{i+2}}=\frac{2\sqrt D}{u_{i+1}c_{i+2}+1}<\sqrt D.$$
\end{proof}

Recursively using the formulas from Proposition \ref{recurrence ni}, one can obtain the desired power series expression for $N_i$ as Theorem \ref{power series ni}. 
The only term of total degree 1 (in $u_0^{-1}, u_1^{-1}, u_2^{-1}, \dots, u_{s-1}^{-1}$) will be $\frac{2\sqrt D}{u_{i+1}}$, which corresponds to the well-known approximation
$N_i\approx\frac{2\sqrt D}{u_{i+1}}$ (see eg. \cite[Proposition 3.3]{Ka}).
In fact, we can improve this estimate as follows (it is an improvement, as $c_{i+1}=u_{i+1}+1/c_{i+2}>u_{i+1}$).

\begin{prop}\label{estimate ni}
For $i\geq 0$ we have
$$\frac{2\sqrt D}{c_{i+1}}\left(1-\frac 1{c_ic_{i+1}}\right)<N_i<\frac{2\sqrt D}{c_{i+1}}.$$
\end{prop}

\begin{proof}
The upper bound follows directly from Proposition \ref{recurrence ni}c), as we have 
$N_i=\frac{2\sqrt D}{c_{i+1}}-\frac{N_{i-1}}{c_{i+1}^2}<\frac{2\sqrt D}{c_{i+1}}.$

To prove the lower bound, we apply the upper bound to Proposition \ref{recurrence ni}c) again
$$N_i=\frac{2\sqrt D}{c_{i+1}}-\frac{N_{i-1}}{c_{i+1}^2}>
\frac{2\sqrt D}{c_{i+1}}-\frac{2\sqrt D}{c_{i}}\cdot\frac{1}{c_{i+1}^2}.$$
\end{proof}

To be able to estimate the error of the estimates, we first need to estimate the errors (coming from truncating the series) when applying the formulas from Proposition \ref{recurrence ni}. This is routine, but somewhat technical.

\begin{lemma}\label{estimate error}

a) For every $k\geq 1$ there is $-1\leq \ve\leq 1$ such that  
$$\frac{1}{c_i}=\sum_{j=0}^{k-1}\frac{(-1)^j}{u_i^{j+1}c_{i+1}^j}+\frac{\ve}{u_i^{k+1}u_{i+1}^k}.$$

b)  
For some $-1\leq \ve\leq 1$ we have
$$\frac{1}{c_i}=\frac{1}{u_i}\left(1-\frac{1}{u_iu_{i+1}}\right)+
\frac{1}{u_i^2u_{i+1}^2}\left(\frac{1}{u_{i}}+\frac{1}{u_{i+2}}\right)\ve.$$

c) For some $-1\leq \ve\leq 1$ we have
$$\frac{1}{c_i^2}=\frac{1}{u_i^2}+\frac{2}{u_i^3u_{i+1}}\cdot\ve.$$

d) For some $-1\leq \ve\leq 1$ we have $$\frac{1}{c_i^2}=\frac{1}{u_i^2}\left(1-\frac{2}{u_iu_{i+1}}\right)+\frac{2}{u_i^4u_{i+1}^2}\cdot\ve.$$
\end{lemma}

\begin{proof}
a) By Proposition \ref{recurrence ni}, we have
$$\frac{1}{c_i}=\sum_{j=0}^{k-1}\frac{(-1)^j}{u_i^{j+1}c_{i+1}^j}\pm\frac{1}{u_i^{k+1}c_{i+1}^k}\cdot \sum_{j=0}^{\infty}\frac{(-1)^j}{u_i^{j}c_{i+1}^j}=
\sum_{j=0}^{k-1}\frac{(-1)^j}{u_i^{j+1}c_{i+1}^j}\pm \frac{1}{u_i^{k+1}c_{i+1}^k}\cdot\frac{1}{1+\frac{1}{u_ic_{i+1}}},$$
and so the error satisfies 
$$\left\lvert\pm\frac{1}{u_i^{k+1}c_{i+1}^k}\cdot\frac{1}{1+\frac{1}{u_ic_{i+1}}}\right\rvert\leq
\frac{1}{u_i^{k+1}u_{i+1}^k}\cdot 1.$$

b) follows from a) with $k=2$ by estimating the term $1/c_{i+1}$ again using a) with $k=1$.

c), d) We have $\frac{1}{c_i}=\frac{1}{u_i}\cdot\frac{1}{1+\frac{1}{u_ic_{i+1}}}$, and so
$$\frac{1}{c_i^2}=\frac{1}{u_i^2}\cdot\frac{1}{1+\frac{2}{u_ic_{i+1}}\left(1+\frac{1}{2u_ic_{i+1}}\right)}=\frac{1}{u_i^2}\sum_{j=0}^{\infty}\frac{(-2)^j}{u_i^jc_{i+1}^j}\left(1+\frac{1}{2u_ic_{i+1}}\right)^j.$$

From here c) and d) follow similarly as in the proof of part a).
\end{proof}

We are finally ready to give the desired expansion for $N_i$.
In the next section we shall need to consider the terms of degree at most 5. Especially when the coefficients $u_j$ are not too small, this gives us very good information on the approximate size of $N_i$. 
On the other hand, when eg. $u_i=u_{i+1}=u_{i+2}=1$, the formulas are nearly useless.

\begin{theorem}\label{power series ni}

a) Degree 1: For some $0< \ve\leq 1$ we have 
$$\frac {N_i}{2\sqrt D}=\frac 1{u_{i+1}}-
\frac 1{u_{i+1}^2}\left(\frac 1{u_i}+\frac 1{u_{i+2}}\right)\ve.$$

b) Degree 3: Assume that $u_{i-1}, u_i, u_{i+1}, u_{i+2}, u_{i+3}\geq u$ for some $u\in\mathbb N$. Then there is some $-1\leq \ve\leq 1$ such that
$$\frac {N_i}{2\sqrt D}=\frac{1}{u_{i+1}}
\left(1-\frac{1}{u_iu_{i+1}}-\frac{1}{u_{i+1}u_{i+2}}\right)+\frac{10}{u^5}\cdot\ve.$$

c) Degree 5: We have
\begin{displaymath}
\begin{split}
N_i= \frac{2\sqrt D}{u_{i+1}}
&\left[  1-\frac{1}{u_iu_{i+1}}-\frac{1}{u_{i+1}u_{i+2}}+
\left(\frac{1}{u_iu_{i+1}}+\frac{1}{u_{i+1}u_{i+2}}\right)^2+ \right. \\
&\left. \ +\frac{1}{u_{i-1}u_{i}^2u_{i+1}}+\frac{1}{u_{i+1}u_{i+2}^2u_{i+3}}\right]+\cdots,
\end{split}
\end{displaymath}
where $\cdots$ stands for terms of total degree greater than 5. 
If $u\in\mathbb N$ is such that $u_j\geq u$ for all $j$, then the error satisfies $|\cdots|<\frac {65}{u^7}$.
\end{theorem}

\begin{proof}
This is a routine repeated application of the formulas from Propositions \ref{recurrence ni}, \ref{estimate ni}, and estimates from Lemma \ref{estimate error}, and so we only illustrate it by proving parts a) and b), i.e., by a computation till degree 3:

a) By Proposition \ref{recurrence ni}c), Lemma \ref{estimate error}a) for $k=1$, and Proposition \ref{estimate ni} we have:
$$\left\lvert\frac {N_i}{2\sqrt D}-\frac 1{u_{i+1}}\right\rvert\leq 
\left\lvert\frac 1{c_{i+1}}-\frac 1{u_{i+1}}\right\rvert+
\left\lvert\frac{N_{i-1}}{2\sqrt D}\cdot\frac{1}{c_{i+1}^2}\right\rvert\leq
\frac 1{u_{i+1}^2u_{i+2}}+ \frac{1}{u_i c_{i+1}^2}\leq \frac 1{u_{i+1}^2u_{i+2}}+\frac 1{u_iu_{i+1}^2}.$$
By Proposition \ref{estimate ni}, we see that the error has to be negative.

b) We first estimate all the $u_j$ in the error terms by $u$, and then repeatedly apply Proposition \ref{recurrence ni}c), Lemma \ref{estimate error}, and part a) of this theorem as follows:
\begin{displaymath}
\begin{split}
\frac{N_i}{2\sqrt D} & = 
\frac{1}{c_{i+1}}-\frac{N_{i-1}}{2\sqrt D}\cdot\frac{1}{c_{i+1}^2}=\\
&= 
\frac{1}{u_{i+1}}\left(1-\frac{1}{u_{i+1}u_{i+2}}\right)+\frac{2\ve_1}{u^5}+
\left(\frac 1{u_{i}}+\frac{2\ve_2}{u^3}\right)\cdot \left(\frac{1}{u_{i+1}^2}+\frac{2\ve_3}{u^4}\right)=\\
&= 
\frac{1}{u_{i+1}}
\left(1-\frac{1}{u_iu_{i+1}}-\frac{1}{u_{i+1}u_{i+2}}\right)+
\frac{2\ve_1}{u^5}+\frac 1{u_{i}}\cdot \frac{2\ve_3}{u^4}+\frac{2\ve_2}{u^3}\cdot\frac{1}{u_{i+1}^2}+ \frac{2\ve_2}{u^3}\cdot\frac{2\ve_3}{u^4},
\end{split}
\end{displaymath}
and we see that the absolute value of the error is less than $\frac{2}{u^5}+\frac{2}{u^5}+\frac{2}{u^5}+\frac{4}{u^7}\leq \frac{10}{u^5}$.

The proof of c) is similar, only more technical. 
\end{proof}

Note that we can also use these results to give a simple proof of Proposition 1 from \cite{JK}, which says that there is an element of norm $\frac{D-a^2}{N}$ for some $a\in\mathbb Z$, $|a|\leq N/2$. Jang-Kim study the prime factorization of $\alpha_{i+1}$ to prove this, but it follows directly by combining Theorem \ref{indecomposable estimate norm} with Proposition \ref{recurrence ni}.

\begin{prop}\cite[Proposition 1]{JK}
Let $i$ be such that $N_{i+1}=N$. Then there is some $0\leq r\leq u_{i+2}$ such that $N(\alpha_{i,r})=\frac{D-a^2}{N}$ for some $a\in\mathbb Z$, $|a|\leq N/2$.
\end{prop}

\begin{proof}
Clearly $i$ is odd, and by Theorem \ref{indecomposable estimate norm} we know that 
$$N(\alpha_{i,r})=\frac{D-(T_{i+1}-rN)^2}{N}.$$
If we take $r_1$ to be the integer for which $|T_{i+1}-rN|$ is minimal, then 
$|T_{i+1}-r_1 N|\leq N/2$ as we want. 
So we only need to check that such $r_1$ lies in the specified interval $0\leq r\leq u_{i+2}$. This will follow if we show that $T_{i+1}-0\cdot N>0>T_{i+1}-u_{i+2} N$, which is easy to see using Proposition \ref{recurrence ni}:

$r=0$: We have $$T_{i+1}-0\cdot N=-\sqrt D+N_{i+1}c_{i+2}=\sqrt D-\frac{N_i}{c_{i+2}}>0,$$
where we first used \ref{recurrence ni}b), then \ref{recurrence ni}c), and finally \ref{recurrence ni}d).

$r=u_{i+2}$: Similarly, we have
$$T_{i+1}-u_{i+2} N=-\sqrt D+N_{i+1}(c_{i+2}-u_{i+2})=-\sqrt D+\frac{N_{i+1}}{c_{i+3}}<0.$$

This finishes the proof (in fact, we shall see in Proposition \ref{heuristic r} that $\frac{u_{i+2}}2-1<r_1<\frac{u_{i+2}}2+1$).
\end{proof}

\section{Norms of indecomposable elements}\label{section counterexample}

In this section we determine for which value of $r$ the indecomposable element $\alpha_{i, r}$ (with $i$ fixed) has maximal norm (Proposition \ref{heuristic r}).
Then we prove a power series formula for this norm, similar to Theorem \ref{power series ni}.
This will give us a heuristic for finding counterexamples to the Conjecture \ref{conjecture} of Jang-Kim. 

\

Assume that $i$ is odd so that $\alpha_{i, r}\succ 0$ is indecomposable for all $0\leq r\leq u_{i+2}$.
Let's first determine which value of $r$ maximizes the norm of $\alpha_{i, r}$.

\begin{prop}\label{heuristic r}
Assume that $i$ is odd and let $r_0$ be such that $N(\alpha_{i,r})$ is maximal among $0\leq r\leq u_{i+2}$. Then 
$\frac{u_{i+2}}2-1<r_0<\frac{u_{i+2}}2+1.$

If $u_{i+2}$ is even, then $$r_0=\frac{u_{i+2}}2.$$
\end{prop}

\begin{proof}
By Theorem \ref{indecomposable estimate norm}a) we have 
$$N(\alpha_{i,r})=\frac{D-(T_{i+1}+rN(\alpha_{i+1}))^2}{|N(\alpha_{i+1})|}=
\frac{D-(T_{i+1}-rN_{i+1})^2}{N_{i+1}},$$
and so the norm is maximal when $|T_{i+1}-rN_{i+1}|$ is minimal, which happens when 
$\left|\frac{T_{i+1}}{N_{i+1}}-r\right|$ is minimal (with $0\leq r\leq u_{i+2}$). 

Let's start by showing that
\begin{equation}\label{eqn:heuristic}
\left|\frac{T_{i+1}}{N_{i+1}}-\frac{u_{i+2}}{2}\right|<\frac 12.\tag{*}
\end{equation}

By Proposition \ref{recurrence ni}b) we have $T_{i+1}=-\sqrt D+N_{i+1}c_{i+2}$, and hence 
\eqref{eqn:heuristic} is equivalent to 
$$N_{i+1}>\left|-2\sqrt D+N_{i+1}(2c_{i+2}-u_{i+2})\right|=
\left|-2\sqrt D+N_{i+1}\left(c_{i+2}+\frac 1{c_{i+3}}\right)\right|.$$
For this, we prove two inequalities:

a) $N_{i+1}>-2\sqrt D+N_{i+1}\left(c_{i+2}+\frac 1{c_{i+3}}\right)$:

By Proposition \ref{estimate ni}, we have 
$$2\sqrt D>N_{i+1}c_{i+2}>N_{i+1}\left(c_{i+2}+\frac 1{c_{i+3}}-1\right),$$
as we wanted to prove.

b) $N_{i+1}>2\sqrt D-N_{i+1}\left(c_{i+2}+\frac 1{c_{i+3}}\right)$:

First note that $$c_{i+1}c_{i+2}=\left(u_{i+1}+\frac 1{c_{i+2}}\right)c_{i+2}=
u_{i+1}c_{i+2}+1\geq c_{i+2}+1,$$
and so $$-\frac 1{c_{i+1}c_{i+2}}+\frac 1{c_{i+2}}-\frac 1{c_{i+1}c_{i+2}^2}\geq 0.$$

Hence
\begin{equation*}
\begin{split}
2\sqrt D & \leq
2\sqrt D-\frac {2\sqrt D}{c_{i+1}c_{i+2}}+\frac {2\sqrt D}{c_{i+2}}-
\frac {2\sqrt D}{c_{i+1}c_{i+2}^2}=
\frac {2\sqrt D}{c_{i+2}}\left(1-\frac 1{c_{i+1}c_{i+2}}\right)(c_{i+2}+1)<\\
& <N_{i+1}(c_{i+2}+1)<N_{i+1}\left(c_{i+2}+1+\frac 1{c_{i+3}}\right),
\end{split}
\end{equation*}
as we wanted to show (note that in the penultimate inequality we used Proposition \ref{estimate ni}).
This proves \eqref{eqn:heuristic}.

\

Let $r_1\in\mathbb Z$ be such that $\left|\frac{T_{i+1}}{N_{i+1}}-r_1\right|$ is minimal.
Then
$\left|\frac{T_{i+1}}{N_{i+1}}-r_1\right|\leq\frac 12,$
and so using the triangle inequality and \eqref{eqn:heuristic},
$$\left|r_1-\frac{u_{i+2}}{2}\right|\leq 
\left|r_1-\frac{T_{i+1}}{N_{i+1}}\right|+\left|\frac{T_{i+1}}{N_{i+1}}-\frac{u_{i+2}}{2}\right|
<\frac 12+\frac 12=1.$$

As $r_1$ is an integer and $u_{i+2}\geq 1$, we see that $0\leq r_1\leq u_{i+2}$, and so $r_0=r_1$ and the proposition is proved.
\end{proof}

\

From now on, we shall assume that $i$ is odd and $u_{i+2}$ even as in Proposition \ref{heuristic r}. Recall that we have defined 
$M_i:=N(\alpha_{i, u_{i+2}/2})$. 
Let's now find a power series expression for the norm $M_i$ till degree $3$, which we shall then use to find the example of Theorem \ref{counterexample}.
(A similar formula holds also in the case of odd $u_{i+2}$, or even for arbitrary $r$, but it's more complicated.)

\begin{prop}\label{power series beta}
Let $i$ be odd and $u_{i+2}$ even. Then

a) Degree 1: 
$$\frac{2M_i}{\sqrt D}=u_{i+2}+\frac 1{u_{i+1}}+\frac 1{u_{i+3}}+\cdots,$$
where $\cdots$ stands for terms of total degree in $u_j^{-1}$ greater than 2.

b) Degree 3: 
$$\frac{2M_i}{\sqrt D}=u_{i+2}+\frac 1{u_{i+1}}+\frac 1{u_{i+3}}-
\left(\frac 1{u_iu_{i+1}^2}+\frac 1{u_{i+3}^2u_{i+4}}+
\frac 1{u_{i+2}}\left(\frac 1{u_{i+1}}-\frac 1{u_{i+3}}\right)^2\right)+\cdots,$$
where $\cdots$ stands for terms of total degree in $u_j^{-1}$ greater than 4.
If $u\in\mathbb N$ is such that $u_j\geq u$ for all $j$, then the error satisfies $|\cdots|<\frac {105}{u^5}$.
\end{prop}

\begin{proof}
First let $0\leq r\leq u_{i+2}$. We have
\begin{displaymath}
\begin{split}
N(\alpha_{i, r})=&(\alpha_i+r\alpha_{i+1})(\alpha_i\pr+r\alpha_{i+1}\pr)=
N(\alpha_i)+r^2N(\alpha_{i+1})+2r(p_ip_{i+1}-Dq_iq_{i+1})\\
=&
N_i-r^2N_{i+1}+2rT_{i+1}.
\end{split}
\end{displaymath}

If we set $r=u_{i+2}$, the left hand side becomes $N_{i+2}$ and we have 
$$2u_{i+2}T_{i+1}=N_{i+2}+u_{i+2}^2N_{i+1}-N_i.$$

Let us now take $r=u_{i+2}/2$ and combine the preceding two formulas to get
\begin{displaymath}
\begin{split}
4M_i=& 4N_i-u_{i+2}^2N_{i+1}+2(N_{i+2}+u_{i+2}^2N_{i+1}-N_i)\\
=& u_{i+2}^2N_{i+1}+2(N_{i+2}+N_i).
\end{split}
\end{displaymath}

Now we divide this equation by $2\sqrt D$ and apply the formulas from Theorem \ref{power series ni}.

a) $$\frac{2M_i}{\sqrt D}=u_{i+2}^2\cdot\frac{1}{u_{i+2}}
\left(1-\frac{1}{u_{i+1}u_{i+2}}-\frac{1}{u_{i+2}u_{i+3}}\right)+
2\left(\frac{1}{u_{i+3}}+\frac{1}{u_{i+1}}\right)+\cdots,$$
which simplifies to the desired formula.

b) is similar, we just use degree 5 expansion for $N_{i+1}$ and degree 3 for $N_i$ and $N_{i+2}$.
\end{proof}

\

Now we are ready to explain how we constructed the counterexample to Conjecture \ref{conjecture} in Theorem \ref{counterexample}. 
The conjecture is based on the expectation that, as $i$ varies, $N(\alpha_{i,r})$ is maximal for $i$ such that $N_{i+1}=|N(\alpha_{i+1})|$ is minimal.
Thus the key is to find some $D$ and odd indices $i<j$ such that 
\begin{itemize}
\item $\alpha_{i+1}$ is the element with the largest negative norm (= the smallest norm in absolute value $N_{i+1}=N$),
\item $N_{i+1}=|N(\alpha_{i+1})|<N_{j+1}=|N(\alpha_{j+1})|$, but the difference of the norms is small,
\item $M_i=N(\alpha_{i, r})<M_j=N(\alpha_{j, t})$ for $r=u_{i+2}/2, t=u_{j+2}/2$ as in Proposition \ref{heuristic r}.
\end{itemize}
We shall do this by prescribing some of the coefficients $u_i$ of the continued fraction for $\sqrt D$ and using Friesen's theorem \cite{Fr} that guarantees the existence of infinitely many such squarefree integers $D$.
Since we are using only heuristics and not precise estimates, we then have to verify that all the conditions are indeed satisfied. Hence in the following discussion we ignore error terms (and place quotation marks around claims that are imprecise).

First of all, if the length $s$ of period of the continued fraction for $\sqrt D$ is odd, then the fundamental unit has the largest negative norm $-1$. To avoid this situation we take $s$ even.

From Theorem \ref{power series ni}a), we see that $N_{i+1}<N_{j+1}$ ``if and only if" $u_{i+2}\geq u_{j+2}$.
But if $u_{i+2}> u_{j+2}$, then Proposition \ref{power series beta}a) ``implies" that $M_i>M_j$, which we don't want. Hence let's take $u_{i+2}=u_{j+2}$ and consider the higher order terms.

Theorem \ref{power series ni}b) then says that $N_{i+1}<N_{j+1}$ ``if and only if" 
$\frac{1}{u_{i+1}}+\frac{1}{u_{i+3}}\geq \frac{1}{u_{j+1}}+\frac{1}{u_{j+3}}$. But again 
Proposition \ref{power series beta}a) ``implies" that if $M_i<M_j$, then strict inequality cannot occur, so that we have $$\frac{1}{u_{i+1}}+\frac{1}{u_{i+3}}= \frac{1}{u_{j+1}}+\frac{1}{u_{j+3}}.$$

In this case \ref{power series ni}c) gives us $N_{i+1}<N_{j+1}$ ``if and only if" 
$$\frac{1}{u_{i}u_{i+1}^2}+\frac{1}{u_{i+3}^2u_{i+4}}\leq\frac{1}{u_{j}u_{j+1}^2}+\frac{1}{u_{j+3}^2u_{j+4}},$$
and \ref{power series beta}b) says $M_i<M_j$ ``if and only if"
$$\frac 1{u_iu_{i+1}^2}+\frac 1{u_{i+3}^2u_{i+4}}+
\frac 1{u_{i+2}}\left(\frac 1{u_{i+1}}-\frac 1{u_{i+3}}\right)^2\geq 
\frac 1{u_ju_{j+1}^2}+\frac 1{u_{j+3}^2u_{j+4}}+
\frac 1{u_{j+2}}\left(\frac 1{u_{j+1}}-\frac 1{u_{j+3}}\right)^2.$$
It seems possible to arrange for both of the last two inequalities to be strict, which should allow us to indeed get $N_{i+1}<N_{j+1}$ and $M_i<M_j$!

First of all, subtracting the inequalities we obtain (note that we're taking $u_{i+2}=u_{j+2}$)
$$\left(\frac 1{u_{i+1}}-\frac 1{u_{i+3}}\right)^2> 
\left(\frac 1{u_{j+1}}-\frac 1{u_{j+3}}\right)^2.$$
Since we also have $$\frac{1}{u_{i+1}}+\frac{1}{u_{i+3}}= \frac{1}{u_{j+1}}+\frac{1}{u_{j+3}},$$
let's take one of the smallest solutions of this system, 
$u_{i+1}=2, u_{i+3}=6$ and $u_{j+1}=u_{j+3}=3.$
Our two inequalities then greatly simplify and we see that $u_i=10, u_{i+4}=1, u_j=u_{j+4}=4, u_{i+2}=u_{j+2}=12$ indeed give a solution with strict inequalities.

We want to place these numbers as coefficients of a continued fraction that isn't unnecessarily long, so take for example $i=1$, $j=7$ and $s=18$, and consider
\begin{equation}\label{eqn:cont frac}
\sqrt D=[u_0, \overline{10,2,12,6,1,3,4,3,12,3,4,2,1,6,12,2,10,2u_0}].\tag{**}
\end{equation}
Note that the sequence $u_1, \dots, u_{s-1}$ is symmetric and that the coefficient $u_6=3$ was chosen experimentally so that everything works nicely. 
By Friesen's theorem \cite{Fr}, we know that there are infinitely many such squarefree integers $D$, so we just find one of them to get Theorem \ref{counterexample}.

\

To conclude, let us sketch the argument that there are most likely infinitely many counterexamples.

For concreteness, we can continue with the example from above and take $\sqrt D$ as in \eqref{eqn:cont frac}. We want to show that there are infinitely many values of $u_0$ such that items (1) -- (5) hold as in Theorem \ref{counterexample}.

It is straightforward to compute that there are infinitely many values of $u_0$ (given by a linear polynomial in a nonnegative integer variable $x$) and $D$ (given by a quadratic polynomial in $x$) satisfying \eqref{eqn:cont frac} -- for the details of this and following computations see the Mathematica notebook at \url{sites.google.com/site/vitakala/research/indec}.

We shall later choose $x$ so that $D\equiv 2\pmod 4$ is squarefree, but first one can formally compute the norms $N_i$ of convergents. These norms are linear polynomials in $x$ and one verifies that $N_2=N$ is minimal (for every $x$). Similarly one computes that $M_7=N(\alpha_{7,6})$ is the largest norm of a semiconvergent.
 
Thus for each such squarefree $D$, items (1), (2), (3), and (5) in Theorem \ref{counterexample} will be satisfied for $0<a_0<\frac N2$ such that $\frac{D-a_0^2}{N}=M_1$. 
But it could happen that this value of $a$ is not the smallest solution of $a^2\equiv D\pmod N$, as required by (4) and Conjecture \ref{conjecture}. 
However, if $N$ is prime, then this congruence has exactly two solutions $0<a_0<N-a_0<N$, and hence (4) is satisfied for $a_0$.

It only remains to arrange for $D\equiv 2\pmod 4$ (which holds when $x\equiv 2\pmod 4$) to be squarefree and, simultaneously, for the value of the linear polynomial $N(x)$ to be prime. 
It is possible to prove this as in \cite{Er} (there is no local obstruction; in fact, in our case $D(x)=N_8(x)M_7(x)$ is the product of two coprime linear polynomials, which are also coprime with $N(x)$).

More generally, one could similarly argue for a general sequence $u_1, \dots, u_{s-1}$ such that the error estimates in 
\ref{power series ni}c) and \ref{power series beta}b) allow one to provably determine the smallest $N_i$ and largest $M_j$.

\end{document}